\documentclass[leqno]{amsart}
\usepackage{amssymb,amsmath,amscd}

\usepackage[english]{babel}

\usepackage{amsthm}
\usepackage{amsfonts}

\newtheorem{thm}{Theorem}[section]

\newtheorem{corollary}[thm]{Corollary}
\newtheorem{prop}[thm]{Proposition}
\newtheorem{lemma}[thm]{Lemma}
\newtheorem{fact}[thm]{Fact}

\theoremstyle{definition}
\newtheorem{defn}[thm]{Definition}
\newtheorem{example}[thm]{Example}

\theoremstyle{remark}
\newtheorem{remark}[thm]{Remark}

\newcommand{\bt}{\begin{thm}}
\newcommand{\et}{\end{thm}}
\newcommand{\bp}{\begin{prop}}
\newcommand{\ep}{\end{prop}}
\newcommand{\bd}{\begin{defn}}
\newcommand{\ed}{\end{defn}}
\newcommand{\bl}{\begin{lemma}}
\newcommand{\el}{\end{lemma}}
\newcommand{\bfa}{\begin{fact}}
\newcommand{\efa}{\end{fact}}
\newcommand{\bc}{\begin{corollary}}
\newcommand{\ec}{\end{corollary}}
\newcommand{\bex}{\begin{example}}
\newcommand{\eex}{\end{example}}
\newcommand{\br}{\begin{remark}}
\newcommand{\er}{\end{remark}}
\newcommand{\ben}{\begin{enumerate}}
\newcommand{\een}{\end{enumerate}}

\newcommand{\C}{\mathbb{C}}
\newcommand{\PP}{\mathbb{P}}
\newcommand{\ds}{{\displaystyle}}
\newcommand{\coo}{{\mathcal O}}

\newcommand{\Ox}{\mathcal{O}_X}

\newcommand{\Z}{\mathbb{Z}}

\newcommand{\Proj}{\mathbb{P}}
\newcommand{\la}{\longrightarrow}

\DeclareMathOperator{\tens}{\otimes}
\newcommand{\sotto}[2]{#1_{#2}}

\newcommand{\rrr}{\rightarrow}
\newcommand{\ra}{\rightarrow}

\newcommand{\ideal}[1]{\sotto {{\mathcal I}}{#1}}

\newcommand{\exact}[3]
{0 \rrr #1 \rrr #2
\rrr #3 \rrr 0}

\begin{document}

\title[A New Curve Algebraically not Rationally Uniformized]{ A New Curve Algebraically but not Rationally Uniformized by Radicals.}
%\date{3 dicembre 2009}
%\shorttitle{A New Curve Algebraically not Rationally Uniformized }
\author{Gian Pietro Pirola, Cecilia Rizzi, Enrico Schlesinger}

\thanks{Partially supported by:
1) MIUR PRIN 2007:
  \textup{Moduli, strutture geometriche e loro applicazioni};
2) Gnsaga.
The second named author was also partially supported by a post-doc scholarship of Poli\-tecni\-co di Milano.}
%\footnotetext{\emph{Mathematics Subject Classification.} Primary 14C25. Secondary 14H40, 14C15}


\maketitle
        \begin{abstract}
We give a new example of a curve $C$ algebraically, but not rationally, uniformized by radicals.
This means that $C$ has no map onto $\PP^1$ with solvable Galois group, while there exists
a curve $C'$ that maps onto $C$ and has a finite morphism to $\PP^1$ with solvable Galois group.
We construct such a curve $C$ of genus $9$ in the second symmetric product of a general curve of genus $2$.
It is also an example of a genus $9$ curve that does not satisfy condition $S(4,2,9)$ of Abramovich and Harris.
        \end{abstract}
\section*{Introduction}
Every smooth projective complex curve $C$ arises as a branched covering of the projective line $\PP^1$, and
its function field is a finite extension of the rational field $\C(x)$. However, it is a difficult problem
to find a method for classifying all possible covering maps $C \ra \PP^1$. As explained by M. Fried in \cite{fried}, Zariski
tackled this problem for the {\em general} curve of genus $g$ (here and in the rest of this paper {\em general} means outside
a countable union of closed subvarieties).  In particular, in \cite{Z} Zariski proves what he regarded
as the analogue for curves of Abel's theorem on the nonsolvability by radicals of a general algebraic equation of degree
$\geq 5$:  the general curve $C$ of genus $g>6$ is not a covering of $\PP^1$ with  solvable Galois group.
Zariski then goes
on and speculates that it would still be possible, though unlikely, that there exist a finite covering $C' \ra C$ with
$C'$ representable as a solvable covering of $\PP^1$. This problem has no analogue in the Galois theory of algebraic equations,
because of the many different ways a curve can be viewed as a covering of $\PP^1$.

To fix the terminology, the Galois group of a branched covering $C \ra C_0$ of smooth curves is the Galois group of the Galois
closure of the finite field extension $\C (C) / \C (C_0)$; it coincides with the monodromy group of the
unramified covering obtained removing the branch divisor from $C_0$. We  say the covering is {\em solvable} if its Galois
group is solvable \cite{harris, PS1}.
A curve $C$ is {\em rationally uniformized by radicals} if there exists a solvable covering map $C \ra \PP^1$,
and  is {\em algebraically uniformized by radicals} is there exists a finite covering $C' \ra C$ with $C'$
rationally uniformized by radicals. Zariski's theorem then says that the general curve of genus $g>6$
is not  rationally uniformized by radicals, and his question is whether $C$ might still be
algebraically uniformized by radicals.  It is not obvious that the two notions are really distinct: the first
example of a curve that is algebraically, but not rationally, uniformized by radicals was given in
\cite{PS} building on work of Debarre and Fahlaoui \cite{DF}. In this paper we give a new example. We feel that it is
of utmost importance to investigate new examples to enhance our understanding of the difficult problem of
describing the possible representations of the general curve $C$ of genus $g$ as a covering of $\PP^1$. Let us recall that
in recent years there has been a lot of research concerning Galois groups of coverings
$C \ra B$ when the genus of $C$ is fixed: see for example \cite{gs} and its list of references. In particular, it is now known that in the moduli space of genus $g$ curves
the locus of curves that are rationally
uniformized by radicals is contained in a proper Zariski closed subset \cite[Theorem 1.6]{n2} and \cite[Theorem 2.4]{gn}.

Debarre and Fahlaoui were motivated by a related problem posed by Abramovich and Harris  \cite{ah}, who formulated
the statement:

\smallskip
\noindent
STATEMENT $S(d,h,g)$: Suppose $C' \ra C$ is a nonconstant map of smooth curves with $C$ of genus $g$.
If $C'$ admits a map of degree $d$ or less to a curve of genus $h$ or less, then so does $C$.
\smallskip

Abramovich and Harris noted that $S(d,0,g)$ is true for elementary reasons, proved $S(2,1,g)$,
$S(3,1,g)$ and $S(4,1,g)$ for $g \neq 7$, and showed that  $S(3,2,5)$ is false. Thus they asked
for which values of $d$, $h$ and $g$ the statement holds. In \cite{DF} it is shown that
$S(4,1,7)$ is false, and in \cite{PS} we showed that the curve of genus $7$ providing the counterexample
is algebraically, but not rationally, uniformized by radicals. Debarre and Fahlaoui construct their
example $C$ on the second symmetric product $Sym^2 (E)$ of an elliptic curve. It is clear by their construction that
$C$ has a covering $C'$ that admits a map of degree $4$ or less to $E$; in particular, $C$ is algebraically uniformized
by radicals. The key step in disproving $S(4,1,7)$ in \cite{DF} is to show
that $C$ has no $4:1$ map to $\PP^1$,
and from this we were able to deduce in \cite{PS} that $C$ is not rationally uniformized by radicals.

In this paper we show the existence of curves  of genus $9$ that are algebraically, but not rationally, uniformized by radicals,
and that provide counterexamples to $S(4,2,9)$. We construct such curves in a linear system $|C|$ on the
second symmetric product $X=Sym^2 (Y)$ of a curve of genus $2$. The hard part of the proof again consists in showing the
general curve in $|C|$  has no $g^1_4$. The technique we use is different from that of \cite{DF} and
comes from an idea of Tyurin \cite{tyurin}. Roughly, the idea is to consider in the Hilbert scheme  $Hilb^4 (X)$
the set $\mathcal{S}$ of all divisors appearing in a $g^1_4$ of a curve in the family $|C|$,
and to bound the dimension of this set in order to show the general curve in $|C|$ cannot contain such a divisor.
To find such a bound, one blows down the canonical divisor of $X=  Sym^2 (Y)$ to obtain the Jacobian surface $S$ of $Y$, and then
notes that the fibers of the Abel sum map $Hilb^4 (S) \rightarrow Alb (S)= S$ are symplectic varieties \cite{Beauville}.
The set $\mathcal{S}$ is contained in such a fiber and, being rationally connected,  must lie in a Lagrangian subvariety.
This provides the desired bound for the dimension of $\mathcal{S}$.

The paper is structured as follows. In Section \ref{prelim} we introduce the notation and terminology, and state some well known
results we will use in the paper. In Section \ref{moduli} we construct the family of curves that will eventually provide the example we are looking for,
and show it has dimension ten. In Section \ref{gonal}, which is the technical heart of the paper, we show that a general curve $C$ in our family
has no $g^1_4$.  In Section \ref{dimcount} we show by a dimension count that a general curve $C$ in a ten dimensional family of genus nine curves
cannot be a covering of $\PP^1$ of degree $d \geq 5$ with a primitive and solvable Galois group. Putting together the results of these two sections
in Section \ref{conclusion} we conclude that $C$ is not rationally uniformized by radicals. On the other hand, by construction every curve in the family is algebraically uniformized by radicals, and we therefore obtain examples of genus $9$ curves that are algebraically,
but not rationally uniformized by radicals.

\section{Notation and Preliminary Results} \label{prelim}
In this section we collect some well known facts that we will use in the sequel of the paper.
We begin recalling the cohomology of divisors of small degree on a curve of genus $2$.

\bp \label{recall} \label{Ycohom}
Let $Y$ be a smooth irreducible projective complex curve of genus $2$, and let $B$ be a divisor on $Y$. Denote by $p \mapsto p'$ the hyperelliptic involution
on $Y$, so that $q=p'$ if and only if $p+q$ is a canonical divisor $K_Y$ on $Y$. Then
\begin{itemize}
    \item[i)] if $\deg(B)=2$, then $h^0(Y, \coo_Y(B))=2$ if $B$ is a canonical divisor, $ h^0 (Y,\coo_Y(B))=1$ otherwise.
    \item[ii)] if $\deg(B)=3$, then $B$ is nonspecial, $h^0(Y, \coo_Y(B))=2$, and
    \begin{itemize}
        \item[a)] either $B \sim K_Y +p$, in which case $p$ is a base point of $|B|$;
        \item[b)] or $B-K_Y$ is not effective, in which case $|B|$ has no base points, and
        $\phi_B: Y \ra \PP^1$ is a morphism of degree $3$.
    \end{itemize}
    \item[iii)] if $\deg(B)=4$, then $B$ is nonspecial, the linear series
    $|B|$ is base point free and defines a morphism $\phi_B: Y \ra \PP^2$. Furthermore:
    \begin{itemize}
        \item[a)] either $B \sim 2K_Y$, in which case $\phi_B$ is the hyperelliptic involution $Y \ra \PP^1$ followed
        by the $2$-uple embedding $\PP^1 \ra \PP^2$; in particular, every effective divisor in $|2K_Y|$ is the sum
        $p+p'+q+q'$ of two canonical divisors;
        \item[b)] or $B \sim K_Y+p+q$ where $q \neq p'$, in which case $\phi_B: Y \ra \PP^2$ is birational onto
        a plane quartic curve, which has a node if $p \neq q$ or a cusp if $p=q$; $\phi_B$ separates any pair of points of $Y$ except for
        the pair $(p,q)$.
    \end{itemize}
\end{itemize}
\ep

We will construct our example on the  second symmetric product $X=Sym^2 (Y)$ of a curve of genus two $Y$.
The surface $X$ has irregularity $q(X)=g(Y)=2$ and geometric genus $\ds p_g(X)= {g(Y) \choose 2}=1$,
hence $\chi (\coo_X)=0$.

We identify points of $X$ with degree $2$ effective divisors $p+q$ on $Y$. The quotient map
$$\pi: Y \times Y \la X,  \qquad \pi(p, q)= p+q$$
exhibits $X$ as the quotient of $Y \times Y$ by the involution $\sigma (p,q)=(q,p)$.
Let $\tilde{\Delta}$ denote the diagonal in $Y \times Y$, and let $\Delta= \pi (\tilde{\Delta})=\{2p: p \in Y \}$.
The map $\pi$ is a double cover ramified along $\Delta$, so that
$\pi_* \coo_{Y \times Y} = \coo_X \oplus \mathcal{L}^{-1}$ where $\mathcal{L}$ a line bundle satisfying
$\mathcal L^{\tens 2}= \Ox(\Delta)$. In particular
$\ds \frac{\Delta}{2}$ is an integral divisor class on $X$, and $\pi^* (\ds \frac{\Delta}{2})= \tilde{\Delta}$.

Given a point $p \in Y$,  we denote by $H_p$ the curve
$$
H_p=\{p+q \in X \,|\, q \in Y \} \subset X.
$$
One knows that the canonical divisors of $Y$ and $X$ are related as follows:
{\em if $p+p'$ is a canonical divisor on $Y$, then  the canonical divisor of $X$ is}
$K_X \sim H_p+H_{p'}-\ds \frac{\Delta}{2}$.
%\begin{proof}
%Since the pull back   $\pi^*: Pic (X) \rightarrow Pic (Y \times Y)$ is injective,
%it is enough to show $$\pi^* \alpha(K_Y)=\pi^* \left(K_X+\frac{\Delta}{2}\right).$$
%
%Now
%$$\pi^*  \alpha(K_Y)= \pi^* (\pi_* pr_1^* (K_Y)) =pr_1^*K_Y +  \sigma (pr_1^*K_Y)=pr_1^*K_Y +  pr_2^* K_Y.$$
%Thus $\pi^*  \alpha(K_Y)$ is the canonical divisor on $Y \times Y$.
%On the other hand
%by adjunction
%$$K_{Y \times Y}= \pi^* K_X + R_{\pi}= \pi^* \left(K_X+\frac{\Delta}{2}\right).$$
%\end{proof}

\medskip
Another way to look at $X$ is via the natural map $X=Div^{(2)}(Y) \rightarrow Pic^{(2)}(Y)$, which exhibits
$X$  as the blow up of the Jacobian variety $Pic^{(2)}(Y)$ at the point corresponding to the canonical divisor $K_Y$.
The exceptional divisor $E \cong \PP^1$ is therefore the unique effective canonical divisor on $X$.
In particular, $E \sim  H_p+H_{p'}-\ds \frac{\Delta}{2}$.

\bigskip
Any divisor $B=  \sum_{p} n_p \, p$ on $Y$ gives rise to a divisor
$\alpha(B)=\sum_p n_p H_p$ on $X$. The map $Pic(Y) \ra Pic(X)$ induced by $\alpha$
is injective, and in fact $Pic (X)$ contains a subgroup isomorphic to
$\alpha (Pic (Y)) \oplus \Z[E]$. Furthermore, a divisor on $X$ is numerically equivalent to zero
if and only if it is linearly equivalent to a divisor of the form $\alpha(B)=\sum_p n_p H_p$ with
$\sum n_p=0$.

When $Y$ has general moduli,
the endomorphism ring of the Jacobian variety $Jac(Y)$ is $\Z$ generated by the identity \cite{koizumi},
and from this it follows $Pic (X)= \alpha (Pic (Y)) \oplus \Z[E]$.

One can easily compute the cohomology of the divisors $\coo_X (\alpha(B))$:
\bp \label{compute}  \
Let $B= \sum_p n_p \, p$ be a divisor on $Y$, and let
$\alpha(B)= \sum_p n_p H_p$ be the corresponding divisor on $X$. Then $\alpha(B)$ is effective if and only if
$B$ is effective. Furthermore
\begin{itemize}
    \item[i)] If $\deg (B)=1$ and $B=p$ is effective, then
     $$h^0(X, \coo_X(H_p))=h^1(X, \coo_X(H_p))=1, \quad  h^2(X, \coo_X(H_p))=0$$
    \item[ii)] if $\deg(B)=2$ and $B \sim K_Y$, then
$$
h^0(X, \coo_X (\alpha(B)))= 3, \quad h^1 (X, \coo_X (\alpha(B)))= 2, \quad h^2(X, \coo_X (\alpha(B)))= 0
$$
    \item[iii)] if $\deg(B)=2$ and $B \nsim K_Y$, then
$$
h^0(X, \coo_X (\alpha(B)))= 1, \quad h^1 (X, \coo_X (\alpha(B)))=h^2(X, \coo_X (\alpha(B)))= 0
$$
   \item[iv)] if $\deg(B)=b \geq 3$, then
$$
h^0(X, \coo_X (\alpha(B)))= \frac{1}{2}b(b-1), \quad h^1 (X, \coo_X (\alpha(B)))= h^2(X, \coo_X (\alpha(B)))= 0
$$
\end{itemize}

\ep
%\begin{proof}
%One can compute all these cohomology groups using
%$$\pi_* \pi^*(\coo_X(D))= \coo_X(D) \oplus \coo_X \left(D+E-\alpha({K_Y})) \right)$$
%together with
%$$
%\pi^*(\coo_X(\alpha(B)))= pr_1^* \coo_Y(B) \otimes pr_2^* \coo_Y(B)$$
%and the Kunneth formula.
%
%\end{proof}

Let $N^1 (X)$ denote the N\'eron-Severi group of $X$ (divisors modulo {\em numerical} equivalence
for which we use the symbol $\equiv$).
We denote by $H$ the class of $H_p$ in $ N^1(X)$, while we keep using the symbols
 $\Delta$ and $E$ for the classes of $\Delta$ and $E$ in $N^1(X)$.
From the above description of $Pic(X)$ we see that $N^1(X)$ contains the subgroup
$$
\Z [H] \oplus \Z [E]
$$
and $N^1(X)= \Z [H] \oplus \Z [E]$ when $Y$ has general moduli. Note that
$$ H^2=1 \qquad H \cdot E=1 \qquad E^2=-1. $$

\bp \ \label{cones} In the N\'eron-Severi group of $X$
\begin{itemize}
%    \item The cone of effective curves in $N^1(X)_{\R}$ is spanned over $\R_+$ by $\Delta$ and $E$.
%     \item The nef cone in $N^1(X)_{\R}$ is spanned by
%     $\ds R_1\equiv
%     H+E$ and $\ds R_2\equiv
%     3H-E$.
 \item  some positive multiple of an integral class $\ds aH - bE$ is effective
      if and only if $a \geq 0$ and $a \geq 2b$.
 \item an integral class $\ds aH-b E$ is ample (respectively nef)
      if and only if $a>-b$ and $a>3b$ (respectively $a \geq -b$ and $a \geq 3b$).

     \end{itemize}
\ep
\begin{proof}
The first statement follows from the fact that $H$ is ample and
the two effective curves $\Delta \equiv 4H-2E$ and $E$ have negative self-intersection.

To check the second statement, let  $ R_1\equiv H+E$ and $ \ds R_2\equiv 3H-E \equiv H+\frac{\Delta}{2}$.
Then $R_1.E=R_2.\Delta=0$. Since $H$ is ample, it follows that $R_1$ and $R_2$ are nef but not ample.
\end{proof}

        \section{A Ten Dimensional Family of Genus Nine Curves} \label{moduli}
In this section we construct a $10$ dimensional family of curves of genus $9$ whose general member we will
eventually show to be algebraically but not rationally uniformized by radicals. We keep the notation
we introduced in the previous section for divisors on the surface $X=Sym^2 (Y)$, where
$Y$ denotes a smooth projective curve of genus $2$.
We begin by showing the existence of smooth genus $9$ curves numerically equivalent to $3H+E$ on $X$.

\bp \label{L}
Let $L$ be a divisor on $X$ numerically equivalent to $3H +E$.
Then
$$h^0(\Ox(L))= 6, \qquad h^1(\Ox(L))=h^2(\Ox(L))=0.$$
Furthermore:
\begin{itemize}
    \item the linear system $|L|$ is base point free, and defines a  morphism $\phi_L: X \ra \PP^5$ that
    maps $X$ birationally onto its image.
    \item the general curve $C \in |L|$ is smooth and irreducible; the genus
of such a curve $C$ is $g(C)=9$ and its self-intersection is $C^2=14$.
\end{itemize}

\ep
\begin{proof}
Any divisor numerically equivalent to $3H$ is ample, hence $h^1(\Ox(L))=h^2(\Ox(L))=0$ by
Kodaira vanishing theorem. Therefore
$$
h^0(\Ox(L))= \chi (\Ox(L))= \frac{1}{2} (L-K_X).L + \chi \coo_X =  6.
$$

Since $L \equiv K_X+3H$ and $H$ is ample, it follows from Reider's Theorem \cite{reider}  that $|L|$ is base point free,
hence the generic curve $C \in |L|$ is smooth (the ground field is $\mathbb{C}$).
Also $L$ itself is ample, hence any curve in $|L|$ is (numerically) connected (cf. \cite{friedman} Ex. 13 p. 24).

We need to check that $\phi_L$ is birational onto its image. The divisor $L$ is linearly equivalent to $E+\alpha (B_0)$ where $B_0$ is a divisor of degree $3$ on $Y$,
thus $\coo_X(L) \cong \coo_X (E+H_p+H_q+H_r)$ where $(p,q,r)$ are three points of $Y$.

By \cite{reider}, Remark 1.2.2, if two points of $X$ are not separated by $|L|$, then there is
a curve $F$ numerically equivalent to $H$ passing through the two points. By Proposition \ref{compute} $F=H_x$ for some $x \in Y$.
We will now show that for every $x$  the linear system $|L|$ separates all but one pair of points of $H_x$, unless
$\coo_Y(p+q+r) \cong \coo_Y(x+K_Y)$ so that $x$ is the unique base point of  $\coo_Y(p+q+r)$. This shows that $\phi_L$ is one to one on $X$ except on an at most one dimensional locus, and concludes the proof.

From the exact sequence
$$
\exact{\coo_X(L-H_x)}{\coo_X (L)}{\coo_{H_x} (L)}
$$
we see every section of $\coo_{H_x} (L)$ arises from a section of $\coo_X (L)$ because $H^1 (\coo_X(L-H_x))=0$ by Kodaira vanishing.
Thus it is enough to show that $\coo_{H_x} (L)$ separates all but a pair of points of $H_x$ if $x$ is not a base point  of $\coo_Y(p+q+r)$.

Thus we assume $x$ is not a base point of $\coo_Y(p+q+r)$, and we can then take  $p$, $q$ and $r$
distinct from $x$.
The restriction $\coo_{H_x} (L)$ of $L$ to the curve  $H_x \cong Y$ is $\coo_Y( x'+p+q+r)$, where
$x'$ is the conjugate point of $x$ so that $x+x' \sim K_Y$. If we had
 $x'+p+q+r \sim 2K_Y$, then $p+q+r \sim x+K_Y$ contradicting the assumption that $x$ is not
 a base point of $\coo_Y(p+q+r)$. Thus $x'+p+q+r \sim 2K_Y$, and
 thus $\coo_{H_x} (L)$ separates all but one pair of points
of $H_x$ by Proposition \ref{Ycohom}.
\end{proof}

We outline now the standard arguments that allow one to compute the dimension of the family of genus $9$ curves $C$ arising as
in the previous proposition.

\bp \label{count}
In the moduli space of curve of genus $9$ there is a $10$ dimensional family of curves whose general member is a curve $C$ numerically equivalent
to $3H+E$ on a surface $X=Sym^2 (Y)$, where $Y$ is a general curve of genus $2$.
\ep

\begin{proof}
Let $f: \mathcal{Y} \ra B$ be a smooth family of genus two curves such that the associated moduli map
$B \ra \mathcal{M}_2$ is generically finite and dominant, and let
$$p: \mathcal{X} = \mathcal{Y} \times_B \mathcal{Y} / \mathcal{S}_2 \ra B$$
be the corresponding family of symmetric products.

Consider the relative Hilbert scheme $\mathcal{H}=Hilb (\mathcal{X}/B)$.
We fix a closed point $b \in B$, and let $X=\mathcal{X}_b $ and $Y=\mathcal{Y}_b$.
Then the fiber $\mathcal{H}_b$  is the Hilbert scheme $Hilb(X)$.

Fix a smooth curve $C\equiv 3H+E$ on $X$.
Since $H^1 (\coo_X (C))=0$, the Hilbert schemes $\mathcal{H}$ and $\mathcal{H}_b$ are smooth at the point
$[C]$ corresponding to $C$, and there exists an exact sequence of tangent spaces
$$
\exact{T_{[C]} \mathcal{H}_b}{T_{[C]} \mathcal{H}}{T_b B}.
$$

\noindent
{\bf Claim}
If $\alpha: \mathcal{H}_b \ra \mathcal{M}_9$ denotes the moduli map, its differential
$$
d \alpha: T_{[C]} \mathcal{H}_b \cong H^0 (C, \coo_C (C)) \ra T_{[C]} \mathcal{M}_b \cong H^1 (C, \mathcal{T}_C)
$$
is injective

\vspace{.3cm}
Since the kernel of $d \alpha$ is (a quotient of) $H^0 (\mathcal{T}_X |_C)$, it is enough
to show the latter group vanishes. For this, we look at the blow up map
$\rho: X \ra S= Pic^2 (Y) \cong Jac (Y)$. Restricting the exact sequence
$$
\exact{\mathcal{T}_X}{\rho^* \mathcal{T}_S \cong \coo_X \oplus \coo_X}{\mathcal{N}_\rho \cong \coo_E(-E) \cong \coo_{\PP^1} (1)}
$$
to the curves $C$ one obtains a new exact sequence
$$
\exact{\mathcal{T}_X |_C}{ (\rho^* \mathcal{T}_S )|_C \cong \coo_C^{\oplus 2} }{\coo_Z }
$$
where $Z$ is the length $2$ zero dimensional intersection of $C$ and $E$.
One checks
$$
H^0 ((\rho^* \mathcal{T}_S )|_C) \ra H^0 (\coo_Z)
$$
is an isomorphism, hence its kernel $H^0 ( \mathcal{T}_X |_C) $ vanishes, proving the claim.

Now let $\beta: \mathcal{H} \ra \mathcal{M}_b$ denote the moduli map on the relative Hilbert scheme. Then $d \beta$
induces a map
$
T_b B \cong H^1 (\mathcal{T}_Y) \ra Coker (d \alpha).
$
Now observe that
$$
Coker (d \alpha) \cong  H^1( \mathcal{T}_C)/Im (H^0 (\coo_C(C))  \hookrightarrow  H^1( (\mathcal{T}_X)|_C)
$$
Thus we obtain a map $$\phi:T_b B \cong H^1 (\mathcal{T}_Y) \ra H^1( (\mathcal{T}_X)|_C).$$

Assume for the moment that $\phi$ is injective. Then $d \beta: T_{[C]} \mathcal{H} \ra T_{[C]} \mathcal{M}_b $
is injective, and this proves the proposition because
$$
h^0 (\coo_C(C))+ h^1 (\mathcal{T}_Y)=7+3=10.
$$

To show $\phi$ is injective, notice that it
factors through the map
$
\psi: H^1 (\mathcal{T}_Y) \ra H^1( \mathcal{T}_X)
$
that associates to a in infinitesimal deformation of $Y$ the corresponding deformation of $X$.
The map $\psi$ is injective: identifying $Y$ with the diagonal $\Delta \subset X$, we see the kernel of
$\psi$ is contained in the kernel of
$
H^1 (\mathcal{T}_\Delta) \ra H^1( (\mathcal{T}_X)|_\Delta),
$
hence in $H^0 (\coo_\Delta (\Delta) )$. The latter group vanishes, hence $\psi$ is injective.

Finally, $\phi$ is obtained composing $\psi$ with
$$
H^1( \mathcal{T}_X) \ra H^1( (\mathcal{T}_X)|_C)
$$
whose kernel is $H^1( \mathcal{T}_X (-C))$. This cohomology group is
contained in $H^1 (\coo_X(-C) ^{\oplus 2})$, which vanishes because $\coo_X(C)$ is ample. Thus  $\phi:  H^1 (\mathcal{T}_Y) \ra H^1( (\mathcal{T}_X)|_C)$ is injective, and this concludes the proof.
\end{proof}

\section{The General Curve in the Family has no $g^1_4$} \label{gonal}
This section contains the main technical difficulty of the paper, which is
to prove that the general curve in our family of genus $9$ curves has no $g^1_4$.
To be more precise, recall that  the Jacobian of a {\em general} smooth projective curve is simple
\cite{koizumi}. Therefore, if $Y$ is a general curve of genus $2$,
the N\'eron Severi group of $X=Sym^2 (Y)$ is generated by the classes of $H$ and $E$. If this is the case
and $L$ is a divisor on $X$ numerically equivalent to $3H+E$, we will show that in the linear system
$|L|$ there is an open dense subset of smooth curves that have no $g^1_4$. Before we can prove this, we need
to establish the fact that these curves have no  map onto a non rational curve.

\bp \label{maps}
Suppose $Y$ is a smooth projective curve of genus $2$ whose Jacobian is simple, and let $X=Sym^2 (Y)$. Suppose
$L$ is a divisor on $X$ numerically equivalent to $3H +E$, and $C$ is a general curve in the linear system $|L|$. If
$D$ is a smooth curve for which there is a finite morphism $C \ra D$
 of degree $d \geq 2$, then $D$ is rational.
\ep
\begin{proof}
Suppose $f: C \ra D$ is a finite morphism of smooth curves of degree $d \geq 2$.
By \cite[Theorem 1.1]{CVV} the Jacobian of $C$ satisfies
$$
End (Jac(C))= \Z \times End (Alb(X))= \Z \times End (Jac (Y))
$$
Since $End(Jac (Y))= \Z$, the abelian subvarieties of $Jac(C)$ have
dimension $0$, $2$, $7$ or $9$. It follows that, if $D$ is not rational and $d \geq 2$, then
$g(D)=2$ and there is an isogeny $\phi: Jac(Y) \ra Jac (D)$ which factors through
the map $Jac(C) \ra Jac(D)$ induced by $f$.

Let $C_0$ be the inverse image of $C$ in $Jac (Y)$ under the map $Jac(Y) \ra Jac (C)$.
Since $\phi: Jac(Y) \ra Jac (D)$ is \'etale, so is its restriction $\psi: C_0 \ra D$. But $\psi$ factors through
$f:C \ra D$, thus $f$ is \'etale.

The family of genus $9$ curves that are \'etale covers of a genus two curve has dimension $3 = \dim \mathcal{M}_2$.
On the other hand, $C$ varies in $|L| \cong \PP^5$, and $|L|$ that by the proof of (\ref{count}) maps with zero dimensional fibers to the moduli space $\mathcal{M}_9$.
Therefore the general $C \in |L|$ is not an \'etale cover of a genus two curve,
finishing the proof of the lemma.
\end{proof}

  \begin{thm} \label{g14}
Suppose $Y$ is a smooth projective curve of genus $2$ whose Jacobian is simple, and let $X=Sym^2 (Y)$. Suppose
$L$ is a divisor on $X$ numerically equivalent to $3H +E$, and $C$ is a general curve in the linear system $|L|$.
Then $C$ has no $g^1_{4}$.
    \end{thm}

\begin{proof}
Suppose by way of contradiction that the general curve in $|L|$ has gonality $d \leq 4$,
and therefore has a base point free $g^1_d$ with $2 \leq d \leq 4$.
We distinguish two cases, according to whether the $g^1_d$ is unique or not.

\bigskip
\noindent {\bf Case 1}
Assume first the general curve $C$ in the family has a unique $g^1_d$.

The natural map $X=Sym^2 (Y) \ra Pic^2 (Y)$ identifies $X$ with the blow up of the abelian
surface $S=Jac (Y)$ at the origin, and $E$ is the exceptional divisor. By a theorem
of Fogarty's \cite{Fogarty} the Hilbert scheme
$\mathcal{H}= Hilb^d (S)$  parametrizing zero dimensional subschemes of $S$ of length $d$
is a smooth and irreducible projective variety of dimension $2d$.
We will identify a zero dimensional subscheme $Z$ of $S$ that does not contain the origin $0_S$, with its preimage
in $X$. Then it makes sense to look at the incidence variety
$$
\left\{
(C, Z) \in |L| \times \mathcal{H}: \; \mbox{$C$ is smooth, \ }, \;
0_S \notin Z, \; Z \subseteq C, \; h^0 (C, \coo_C (Z)) \geq 2
\right\}.
$$

Let $W$ be an irreducible component of this locus that maps dominantly to $|L|$, and let
$\pi_1$ and $\pi_2$ be the two projections of $W$ on $|L|$ and
$\mathcal{H}$ respectively. Since  the general curve in $|L|$ has a unique $g^1_d$, the map $\pi_1$ is dominant
and its  general fiber is a rational curve. Therefore $W$ is a rationally connected variety of dimension $6$
\cite{ghs,ds}.

%The Albanese variety
%of $S$ is isomorphic to the Jacobian $S$ of $Y$, and the Albanese map induces
%a morphism
Now look at the Abel sum map
$$
\alpha: \mathcal{H}= Hilb^d (S) \rightarrow Alb (S)= S.
$$
Since $W$ is rationally connected, the image $\pi_2 (W)$ of $W$ in $\mathcal{H}$ must
be contained in a fiber $K$ of $\alpha$. The fiber $K$ is a symplectic variety of dimension $2(d-1)$- see \cite{Beauville} .
The pull back of the symplectic form to $W$
vanishes because $W$ is rationally connected, hence $\pi_2 (W)$ is a Lagrangian subvariety of $K$.  Therefore
$$\dim \pi_2 (W) \leq \frac{1}{2} \dim K= d-1,$$
and the generic fiber of $\pi_2$ has dimension at least $7-d$.

By Proposition \ref{L} the linear system $|L|$ defines a morphism $\phi: X \ra \PP^5$ that
maps $X$ birationally onto its image. Note that $\PP^5$ is the dual projective space of $|L| \cong \PP (H^0 (X, L)) $.
For a closed subscheme $V \subset X$, we let $|L(-V)| \cong \PP( H^0 (X, L \otimes \ideal{V}))$
denote the linear system of curves in $|L|$ that contain $V$. Then the {\em linear span} of
$\phi(V)$ in $\PP^5$ is the subspace dual to $|L(-V)|$; the dimension of the linear span
of $\phi(V)$ is therefore $4- \dim |L(-V)|$.

Now let $(C, Z)$ be a point of $W$. Then $|L(-Z)|$ contains the fiber $\pi_2^{-1}(Z)$, hence
$$
\dim |L(-Z)| \geq \dim (\pi_2^{-1}(Z)) \geq 7-d.
$$
It follows that  the image $\phi(Z)$ of $Z$ in $\PP^5$ is contained in a linear space of dimension
$d-3$. Since $d \leq 4$, this says that $\phi(Z)$ is contained in a line. Since $\phi$ is birational,
for $Z$ general the linear span of $\phi(Z)$ will be a line, hence $d=4$,
and $\dim |L(-Z)| =3$.
 Since $\dim (\pi_2^{-1}(Z)) \geq 3$, we conclude
that $\dim (\pi_2^{-1}(Z))=3$ for a general $Z$ in $\pi_2 (W)$, and therefore
the general curve $C$ in $|L(-Z)|$  belongs to $\pi_2^{-1}(Z)$, that is, $h^0 (\coo_C (Z)) \geq 2$. By semicontinuity,
$h^0 (\coo_C (Z)) \geq 2$ for {\em every} smooth $C$ in $|L(-Z)|$.

For a point $a \in X$ define
\begin{eqnarray*}
B^0_a=&
\left\{
x \in X: \mbox{there exists $Z \in \pi_2(W)$ such that $a,x \in Z$ } \right.
\\
&
\left.
\mbox{and $h^0 \coo_C (Z) \geq 2$ for every smooth $C$ in $|L(-Z)|$}
\right\}.
\end{eqnarray*}
and let $B_a$ denote the closure of $B^0_a$ in $X$. For a general choice of $a$, the dimension
of $B_a$ is one. To see this, let $W_a$ denote the set of pairs $(C, Z)$ in $W$ for which $a \in Z$.
Then $W_a$ has dimension $4$  because $\pi_1$ maps $W_a$ generically onto  $|L(-a)| \cong \PP^4$,
with zero dimensional fibers as there is a unique divisor in the $g^1_4$
of $C$ that contains $a$. Since the general fiber of $\pi_2: W_a \ra \mathcal{H}$ has dimension $3$,
the image of $W_a$ in the Hilbert scheme is a curve $T$. Therefore the restriction $U_T$
to $T$ of the universal family over $\mathcal{H}$ is also a curve, and so is
$B_a$ which is the closure of the projection in $X$ of $U_T$ with the point $a$ removed.

We claim that, for a general $C_0 \in |L(-a)|$,
$$B_a. C_0 = ma+x_1+x_2+x_3$$
where $Z_0=a+x_1+x_2+x_3$ is the unique element of the $g^1_4$ of $C_0$ that contains $a$,
and $m \geq 0$. Indeed, since $C_0$ is general in $|L(-a)|$, it is smooth and it does not contain
any of the finitely many points of $B_a -B^0_a$ except perhaps $a$. So, if $x \in B_a \cap C_0 \setminus \{a\}$,
there is  $Z \in \pi_2(W)$ such that $a,x \in Z$ and $h^0 \coo_{C} (Z) \geq 2$
for every smooth $C \in |L(-Z)|$. Now
$$
3 \leq \dim |L(-Z)| = \dim |L-(a+x)|
$$
hence $|L(-Z)| = |L-(a+x)|$ and $C_0 \in |L(-Z)| $. Then
$h^0 \coo_{C_0} (Z) \geq 2$, and
$Z=Z_0$ because there is a unique divisor in the $g^1_4$
of $C_0$ that contains $a$. In particular, $x \in \{x_1,x_2,x_3\}$ proving our claim.

The claim implies that the intersection of
$\phi(B_a)$ with a general hyperplane  of $\PP^5$ through $\phi(a)$ is contained
in a line: the hyperplane corresponds to $C_0$, and the line is the linear span of $\phi(Z_0)$.
It follows that $\phi(B_a)$ is contained in a $\PP^2$ through $\phi(a)$, hence
$$
\dim |L(-B_a)| \geq 2
$$

We conclude that $B_a$ and $C-B_a$ are effective, with $h^0 \coo_X(C-B_a) \geq 3$.
Now we use the fact that the N\'eron Severi group of $X$ is generated by $E$ and $H$.
Since $B_a$ moves with $a$, we see $B_a \neq E,2E$. On the other hand,
$C-B_a$ can't be numerically equivalent to $H+2E$, $H+E$ or $H$ because $h^0 \coo_X(C-B_a) \geq 3$.
It then follows from
Proposition \ref{cones} that
$C-B_a$ is numerically equivalent to either $2H$ or $2H+E$, so that either
$B_a \equiv H$ or $B_a \equiv H+E$.

Suppose $B_a \equiv H$. Then $B_a$ is one of the curves $H_p$ (with $p \in Y$). As $B_a.C=H.C =4$, the $g^1_4$ on $C$ is
$|H_{C}|$, where $H_C=H_p .C$. Now look at the exact  sequence
$$
\exact{\coo_X(H_p-C)}{\coo_X (H_p)}{\coo_C (H_C)}
$$

Now $C-H \equiv 2H+E= 2H +K_X$ is ample by \ref{cones}, hence $H^1 \coo_X(H_p-C)=0$,
therefore $h^0 (C, \coo_C (H_C))=1$, so $H_C$ cannot be a pencil, and this case does not occur.

Suppose now $B_a \equiv H+E$, that is, $B_a= H_p +E$ for some $p \in Y$. Then $$B_a.C=(H+E).C=4+2=6$$
contains the $g^1_4$, hence $h^0 ((H_p+E)_C) \geq 2$.
But from the exact sequence
$$
\exact{\coo_X(H_p+E-C)}{\coo_X (H_p+E)}{\coo_C ((H+E)_C)}
$$
one obtains a contradiction as above (note that $h^0 (\coo_X(H_p+E))= \chi (\coo_X(H_p+E))$
because $H_p+E=H_p+K_X$, and $\chi (\coo_X(H_p+E))= \ds \frac{1}{2} (H+E).H=1$).

\vspace{.2cm}
\noindent
{\bf Case 2} \ Suppose now the general $C \in |L|$ has more than one $g^1_d$.
Then $d =4$ because an hyperelliptic curve has a unique $g^1_2$, and a trigonal
curve of genus $g >4$ has a unique $g^1_3$.
So suppose the general $C \in |L|$ has gonality $4$ and has two distinct $g^1_4$.
The two $g^1_4$ define a morphism
$$
\psi: C \ra Q=\PP^1 \times \PP^1$$
of degree $e$ onto a a divisor $B$ of type $(4/e, 4/e)$ on $Q$.
We cannot have $e=4$, as otherwise $B \cong \PP^1$ and the two linear series coincide.
If $e=2$, then $B$ cannot be rational because $C$ is not hyperelliptic. But
$B$ cannot be an elliptic curve because $C$ does not have morphism to curves
of genus $1$ by Proposition \ref{maps}. Hence $e=1$ and $\psi:C \ra B$ is birational.

Since $B$ has arithmetic genus $9$, the map $\psi$ is an isomorphism. Thus $B \cong C$ is a smooth curve of type $(4,4)$ on $Q=\PP^1 \times \PP^1$. We will show that this is not possible by proving that: (a) the curve $C$ has a base point free complete $g^2_8$ that defines a  map
$\phi: C \ra \PP^2$ that is not $4:1$ onto its image; (b)  the only base point free and complete
$g^2_8$  on a divisor $B$ of type $(4,4)$ on $Q$  are $|\coo_B(2,0)|$ and $|\coo_B(0,2)|$, and these define $4:1$ maps.
%and those obtained
%as subseries of $|M|=|\coo_B (1,1)|$.

We now show a general $C$ in $|L|$ has a complete and base point free $g^2_8$ that defines a morphism
$g:C \ra \PP^2$ that is not $4:1$ onto its image.
Recall $X$ is the blow up of the abelian surface $S= Jac(Y)\cong Pic^2(Y)$ at the origin $0_S \cong \coo_Y(K_Y)$. Given a point $p$ in $Y$ we denote
by $\theta_p$ the theta divisor
$$
\theta_p= \{\coo_Y(p+y) \,|\, y \in Y \}.
 %qqqqqqqqqqui
$$
If $p+p'=K_Y$, the divisor $\theta_p+ \theta_{p'}$ is symmetric with respect to the involution of the abelian
surface, and defines a morphism $S \ra \PP^3$ whose image is a quartic Kummer surface $T$. The pull back
of $\theta_p+ \theta_{p'}$ to $X$ is the divisor
$H_p+H_{p'}+2E$, which therefore defines a $2:1$ morphism $X \ra \PP^3$ whose image is the Kummer surface $T$.
This morphism maps $C$ birationally onto its image in $T$ because
$C$ has no $2:1$ morphism to a curve, as it is not hyperelliptic and does not have any morphism of degree
$d \geq 2$ onto a nonrational curve.

Next we project the Kummer surface from the node that is the image of $0_S \cong \coo_Y(K_Y)$: this amounts
to consider the morphism $X \ra \PP^2$ defined by $H_p+H_{p'}$ (one can check
$h^0\left(X,\coo_X\left(H_p+H_{p'}\right)\right)$ $=3$
using Proposition \ref{compute}.ii).
This projection is
a degree $2$ morphism $\gamma: T \ra \PP^2$. Since $C$ maps birationally onto its image in $T$,
the restriction  $g: C \ra \PP^2$ of $\gamma$ to $C$ is either birational or $2:1$ onto its image, in any case is not $4:1$.

To finish, observe that $g^*( \coo_{\PP^2}(1))$ is the line bundle corresponding to the divisor
$\left(H_p+H_{p'}\right)|_C$. We claim $\left(H_p+H_{p'}\right)|_C$ moves in a complete $g^2_8$ on $C$.
To compute $\coo_C \left(H_p+H_{p'}\right)|_C$, we recall
$
h^0\left(X,\coo_X\left(H_p+H_{p'}\right)\right) =3
$ and look at the exact sequence
$$
\exact{\coo_X\left(H_p+H_{p'}-C\right)}{\coo_X \left(H_p+H_{p'}\right)}{\coo_C \left(\left(H_p+H_{p'}\right)|_{C}\right)}
$$
Since $C -H_p -H_{p'} \equiv H+E$ is big and nef by \ref{cones}, the $H^0$ and $H^1$ of $H_p+H_{p'}-C$ vanish,
hence
$$
h^0\left(C, \coo_C \left(\left(H_p+H_{p'}\right)|_{C}\right)\right)=3.
$$

Summing up we have found a complete $g^2_8$ on $C$ that defines a morphism $g: C \ra \PP^2$ that is not $4:1$ onto its image.

\bigskip
To complete the proof, we need to show that, if $B$ is a smooth divisor of type $(4,4)$ on the quadric surface $Q \subset \PP^3$,
then the only base point free and complete
$g^2_8$  on  $B$  are $|\coo_B(2,0)|$ and $|\coo_B(0,2)|$.
So let $\delta$ be an effective divisor of degree $8$ on $B$ such that $|\delta|$ is a complete and base point free $g_8^2$. Then
$h^0 (B, \coo_B(\delta))=3$ and we may assume that $\delta$ consists of $8$ distinct points.

By Riemann-Roch
$$
h^0 (B, \coo_B (K_B-\delta))=3+\deg(K_B-\delta)+1-g(B)=3
$$
Since $B$ is a divisor of type $(4,4)$, by adjunction
$K_B=2H$, where $\coo_B(H)=\coo_B(1,1)$. Therefore $\coo_B(K_B-\delta)=\ideal{\delta,B} (2H)$, and
$$
h^0 (Q, \ideal{\delta,Q}(2,2)) = h^0 (B, \ideal{\delta,B} (2H))=3
$$
because $H^0 (Q, \ideal{B,Q})=H^1 (Q, \ideal{B,Q})= 0$. Thus we see that the linear system
$$\mathcal{D}=|\coo_Q(2,2)-\delta|$$
cut out on $Q$ by quadric surfaces containing $\delta$ has projective dimension $2$.

If the linear system $\mathcal{D}$ had no fixed component, then two general elements
in the linear system would meet properly in a zero dimensional scheme of degree
$8$ containing $\delta$, hence equal to $\delta$. Thus $\delta$ would be a complete intersection of three quadrics in $\PP^3$,
contradicting  $h^0 (Q,\ideal{\delta,Q}(2,2))=3$.
Thus $\mathcal{D}$ has a fixed curve, say $D$. We write $\delta= \alpha + \beta$
where  $\alpha$ consists of those points of $\delta$ that are in the support of $D$.

\vspace{.2cm}
\noindent \textbf{Case a: } $D$ is a line, with respect to the embedding of $Q$ in $\PP^3$ by $\coo_Q(1,1)$, say of type $(0,1)$. \

\vspace{.2cm}
Then the linear system $\mathcal{D}$ is, up to removing $D$, the linear system
$$
|\coo_{Q} (2,1)-\beta|
$$

Observe that $\beta$ cannot contain $3$ collinear points, otherwise the
line through them would be contained in the fixed component of $\mathcal{D}$. Similarly,
if $\beta$ had $4$ points in a plane $\Pi$, then there would be a fixed component of
$\mathcal{D}$ contained in $\Pi$ and different from $D$, which is absurd.

Since $D$ is a line and $\alpha \subset D \cap B$, there are at most four points in $\alpha$,
so $\beta$  contains at least $4$ points $P_1$, $P_2$ $P_3$ and $P_4$. As $\beta$ has no $3$ collinear points, and no $4$ coplanar points,
the points $P_i$ impose independent conditions on $|\coo_Q(2,1)|$:
a divisor of type $(2,1)$ containing only the first three points is given by a conic through $P_1$, $P_2$, and $P_3$, plus a line not containing
$P_4$. Then
$\dim \mathcal{D}  \leq \dim |\coo_Q(2,1)|-4= 1$, a contradiction.

Thus this case does not occur.

\vspace{.2cm}
\noindent \textbf{Case b: } $D$ is a plane section of $Q$, that is a divisor of type $(1,1)$. \

\vspace{.2cm}
Then the linear system $\mathcal{D}$ is, up to removing $D$, the linear system
$$
|\coo_{Q} (1,1)-\beta|
$$

Since $\mathcal{D}$ has projective dimension $2$, we see that $\beta$ consists of at most one
point. If $\beta=P$ had degree one, then $\alpha$ would have degree $7$ besides being contained
in the plane of $D$. Thus
$$h^0 (B, \coo(\alpha)) \geq h^0 (B, \coo_B(1,1)) -1= 3$$

But $\alpha \leq \delta$ and $h^0 (B, \coo_B(\delta))=3$ by assumption,
so the unique point of $\beta$ is a base point of $\delta$, contradicting
the fact that $\delta$ is base point free. Thus $\beta$ is empty. But then
$\delta=\alpha$ is a plane section, and this is also a contradiction because
$h^0 (B, \coo_B(1,1))=4$.  Thus this case also does not occur.

\vspace{.2cm}
\noindent \textbf{Case c: } $D$ has degree $3$ and arithmetic genus $0$,
 that is, it is a divisor of type $(2,1)$ or $(1,2)$. \

\vspace{.2cm}
Suppose that $D$ has type $(1,2)$.
Then the linear system $\mathcal{D}$ is, up to removing $D$, the linear system
$$
|\coo_{Q} (1,0)-\beta|
$$
Since $|\coo_{Q}(1,0)|$ has projective dimension $1$, this case does not occur.

\vspace{.2cm}
\noindent \textbf{Case d: } $D$ has degree $2$ and arithmetic genus $-1$,
 that is, it is a divisor of type $(2,0)$ or $(0,2)$. \

\vspace{.2cm}
Suppose that $D$ has type $(0,2)$. Note that $D$ is either the disjoint union of two lines
of type $(0,1)$, or a double structure on a line of type $(0,1)$.

The the linear system $\mathcal{D}$ is, up to removing $D$, the linear system
$$
|\coo_{Q} (2,0)-\beta|
$$
Since $|\coo_{Q}(2,0)|$ has projective dimension $2$ and no base points, $\beta$ is the zero
divisor. Thus $\delta=\alpha$ is contained in the fixed curve $D$. We know that at most four points of $\delta$ are collinear, because $\delta$ lies on the curve $B$ which has type $(4,4)$. Therefore the only possibility is that $D$ is the union of two lines $L_1$ and $L_2$
of type $(0,1)$, and $\delta= B.D$. Therefore $\coo_B(\delta) \cong \coo_B(0,2)$
(or $\coo_B(\delta) \cong \coo_B(2,0)$ if $D$ has type $(2,0)$).

There are no other possibilities for $D$, because it is contained properly in a divisor of type $(2,2)$. Thus we have proven that the only complete and base point free $g_8^2$'s on $B$ are $\coo_B(2,0)$ and $\coo_B(0,2)$.
\end{proof}

\section{Number of Moduli of Genus $9$ Primitive and Solvable Coverings of the Projective Line} \label{dimcount}
In this section we show that, if $C$ is a general curve in a $10$ dimensional family of smooth curves of genus $9$,
then a finite map $f: C \la \PP^1$ that has a  primitive and solvable Galois group has degree at most $4$
(the Galois group is primitive when $f$ cannot be factored nontrivially). The proof is based on a counting argument due to Zariski \cite{Z}.

Given a finite morphism $f: C \la C_{0}$ of smooth curves,
we denote by $b(q)$ the multiplicity of a branch point $q$ of $f$ in the branch divisor.

\begin{thm}[Zariski, see {\cite[Proposition 3.1]{PS}}]\label{zar}
Let $f:C\la C_{0}$ be a degree $d$ primitive solvable covering of curves. Then there exists a prime $p$ such that $d=p^k$, and
for every branch point $q$ of $f$ the multiplicity $b(q)$ is bounded by the formula:
$$b(q) \geq \frac{p^k-p^{k-1}}{2}.$$
Moreover, if $p=2$ and $d-1$ is prime, then $b(q)\geq 2^{k-1}-1$.
\end{thm}
We recall Zariski's argument.
Let $S_d= Aut(\Omega )$, where $\Omega $ is a set of $d$ elements.
Fix $x \in \Omega$ and consider a primitive solvable subgroup $G$ of $S_d$. It is well known that $G$
has only one minimal normal subgroup $A$, which is an elementary abelian $p$-group for some prime $p$. Moreover, $G$ is the semidirect product
$G=[A] \cdot G_x$, where $G_x$ denotes the stabilizer of $x$ in $G$. Since the action of $A$ on $\Omega$ is regular, the cardinality of $A$ is equal to $d=p^k$ for some $k \geq 1$. Furthermore, identifying $\Omega$ with the vector space $A$,
the group $G$ acts as a subgroup of the group $\mbox{Aff}(A)$ of affinities of $A$. Hence an element $g\in G$, not equal to the identity $1_G$, has at most $p^{k-1}$ fixed points.
To each branch point $q$ the monodromy representation associates an element $g \in G$ whose action on $\Omega$ has $d-b(q)$ orbits.
Hence $\ds b(q) \geq \frac{d-n}{2}$ where $n$ is the number of fixed points of $g$.

\bp \label{rational}
Let $\mathcal C$ be an irreducible family of curves of genus $9$ whose general curve is a degree $d$ primitive and solvable covering of $\PP^1$.
Then the dimension of $\mathcal C$ is at most $9$ unless $d \leq 5$.
\ep

\begin{proof}
Consider a family $\mathcal C$ of curves $C$ of genus $9$ such that the general curve $C$ admits a degree $d$ primitive solvable covering of
 $\PP^1$ with $r$ distinct branch points, each with multiplicity at least $m$.

 By the Riemann Hurwitz
 formula  the degree of the branch divisor $B$ of the covering $f:C \ra \PP^1$ is
$16+2d$.
If  $f$ has exactly $r$ distinct branch points, each with multiplicity at least $m$, then $deg(B) \geq rm$, therefore
\begin{equation*}%\label{mult}
r \leq \frac{16+2d}{m}
\end{equation*}
%\leq 4 \;\frac{8+p^k(g(C_{0})-1)}{p^k-p^{k-1}}.$$
The dimension of $\mathcal C$ is then at most
$$r-3 \leq \frac{16+2d}{m} - 3.$$
By (\ref{zar}) $d=p^k$ and
$$m \geq \frac{p^k-p^{k-1}}{2}$$
%(respectively   $m\geq 2^{k-1}-1$ if $p=2$ and $d-1$ is prime).
Thus

$$\dim \mathcal{C} \leq  4 \, \frac{8+p^k}{p^k-p^{k-1}}-3= \frac{32}{p^{k-1}(p-1)} + \frac{4}{1-1/p}-3.$$

This shows $\dim \mathcal{C} \leq  9$ unless $d \leq 5$ or $d=8$. But when $d=8$ we can use the better estimate
$m\geq 2^{k-1}-1=3$ to conclude $\dim \mathcal{C} <  8$.
\end{proof}

    \begin{prop} \label{five}
If $C$ is a general curve in a $10$ dimensional family of smooth genus $9$ curves, then there is no degree $5$ covering $C \ra \Proj^1$ with a primitive solvable Galois group.
    \end{prop}

    \begin{proof}
Let $f: C \ra \Proj^1$ be a degree $5$ primitive and solvable covering. By Riemann-Hurwitz the branch divisor $B_f$ has degree $26$.
By (\ref{zar})  every branch point $q$ has multiplicity $b(q) \geq 2$, hence the number
$r$ of distinct branch points has to be $\leq 13.$ \\
The Hurwitz scheme of coverings of $\PP^1$ having $r$ distinct branch points %each with multiplicity at least $2$
has dimension $r-3$. Thus, if $C$ varies in a family of dimension $\geq 10$, the only possibility is that $r=13$ and $b(q)=2$ for every
branch point $q$.
%So $B_f= 2 p_1 + 2p_2+ \cdots+2p_{13}$.\\
Let $G\subset S_5$ be the Galois group of $f$. In the Galois group, for any
branch point $q_i \in B_f$, there is an associated cycle
$g_i\in G$ with $3=d-b(q_i)$ orbits. Furthermore by Zariski's argument the cycle $g_i$ has at most one fixed point,
so it must be $(12)(34)$ up to conjugation.
Moreover, the product of the cycles $g_i$ is the identity of the Galois group.
\begin{equation}\label{prod}
\prod_{i=1}^{13} g_i= {1}_{G}.
\end{equation}

As explained after Theorem \ref{zar}, the group $G$ is contained in the group of affinities of $\mathbb Z_5$, which is a semidirect product
$[\mathbb Z_5]\cdot\mathbb Z_4$. Therefore
there is an induced map $\phi: G \rightarrow \mathbb{Z}_4$. Since $g_i$ has order $2$, it can't be contained
in the kernel of $\phi$, so $\phi(g_i)$ is the unique element $h=[2]$ of order $2$ in $\mathbb Z_4$.
Then $\phi (\prod_{i=1}^{13}g_i)$ is also equal to $h$, contradicting (\ref{prod}).
    \end{proof}

\section{Proof of the Main Result} \label{conclusion}
In this section we collect all the previous results to show that the general curve of the family constructed in section
\ref{moduli} is algebraically, but non rationally, uniformized by radicals.

\bt \label{main}
Let $\mathcal{C}$ be the family of smooth genus $9$ curves $C$ for which there exists
a genus $2$ curve $Y$ such that $C \subset   X=Sym^2 (Y)$ and the class of $C$ in the
N\'eron-Severi group of $X$ is $3H +E$. Then a general curve in $\mathcal{C}$ is
algebraically, but non rationally, uniformized by radicals, and provides a counterexample
to Statement $S(4,2,9)$ of \cite{ah}.
\et

\begin{proof}
To see that a curve $C$ in the family is algebraically uniformized by radicals is easy.
Since $C \subset X=Sym^2(Y)$ and $C \equiv 3H+E$, for every point $p \in Y$ the curve $C$ intersects the divisor
$H_p \subset X$ in a scheme of length $4$. We define a map $\phi: Y \la Sym^4(C)$ sending a point $p \in Y$ to
$H_p \cdot C$. Then as in \cite[Proposition 5.1 ]{PS} we deduce there is a smooth curve $C'$ that covers $C$ and
admits a morphism $C' \ra Y$ of degree $\leq 4$.

We now have to prove show that a general $C$ in our family is not rationally uniformized by radicals, that is, there does not exist
a finite map $C \la \mathbb{P}^1$ with solvable Galois group. In particular, $C$, contrary to its covering $C'$,
does not admit a nonconstant map of degree $4$ or less to a curve of genus $2$ or less, and thus provides a counterexample to
$S(4,2,9)$.

Suppose by way of contradiction there is $C \la \mathbb{P}^1$ with solvable Galois group. Then we can factor it as
$$\begin{array}{ccccc}
C \;\;\stackrel{f}{\la}\;\; C_{0}\\
\;\;{\searrow}\;\;\;  \swarrow g &  \\
\Proj^1
\end{array}$$
where $f$ is a covering of degree $d \geq 2$ with a  primitive and solvable Galois group (the Galois group is primitive when $f$ cannot be factored).
Since $C$ does not cover any non rational curve by Proposition \ref{maps}, the curve
 $C_{0}$ is rational, and we are reduced to show   there does not exist
a finite map $C \la \mathbb{P}^1$ with primitive and solvable Galois group.
By the dimension count of Propositions \ref{rational} and \ref{five} any such map would have degree $d \leq 4$. But by Theorem \ref{g14}
there are no morphisms $C \ra \PP^1$ of degree $d \leq 4$. Thus the proof is complete.

%%%Ciao Pietro,
%%%
%%%ho ricontrollato algebricamente il tuo conto per la g2_8 su nostra curva.
%%%
%%%La g2_8 e' H_p+H_q, dove
%%%p+q= canonico sulla curva Y di genere 2.
%%%
%%%Allora se X e' il prodotto simmetrico
%%%H_p+H_q= K_X + delta/2.
%%%
%%%Quindi h0 (H_p+H_q)=h2 (-delta/2),
%%quest'ultimo si puo' calcolare usando il rivestimento doppio
%%%pi: S=YxY ->X=Y^(2) per il quale
%%%
%%%pi_* O_S= O_X+ O_X(-delta/2).
%%%
%%%per cui
%%%
\end{proof}

\end{document}